\documentclass{amsart}

\usepackage{amssymb}
\usepackage{mathtools}
\usepackage{graphicx}
\usepackage{url}
\usepackage{hyperref}
\usepackage{placeins}

\makeatletter
\renewcommand{\@defaultbiblabelstyle}[1]{[#1]}
\makeatother

\newcommand{\Fil}[1]{\textsc{Fil}^*\left(#1\right)}
\newcommand{\Conv}{\mathsf{Conv}}
\newcommand{\Top}{\mathsf{Top}}
\newcommand{\Lim}{\mathsf{Lim}}
\newcommand{\Kent}{\mathsf{Kent}}

\newcommand{\apair}[1]{\left\langle #1\right\rangle}
\newcommand{\NETs}[1]{\mathsf{NETS}\left(#1\right)}

\newcommand{\dom}[1]{\operatorname{dom}\left(#1\right)}

\newtheorem{theorem}{Theorem}[section]
\newtheorem{lemma}[theorem]{Lemma}
\newtheorem{proposition}[theorem]{Proposition}

\theoremstyle{definition}
\newtheorem{example}[theorem]{Example}
\theoremstyle{remark}
\newtheorem{remark}[theorem]{Remark}
\numberwithin{equation}{section}

\begin{document}

\title{Pasting and stability in convergence spaces}

\author{Renan M. Mezabarba}
\address{Departamento de Ci\^encias Exatas, Universidade Estadual de Santa Cruz, Ilh\'eus, BA, 45662-900}
\email{rmmezabarba@uesc.br}

\author{Rodrigo S. Monteiro}
\address{Instituto de Ci\^encias Matem\'aticas e de Computa\c{c}\~ao, Universidade de S\~ao Paulo, S\~ao Carlos, SP, 13566-590}
\email{rodrigosm@usp.br}

\author{Thales F. V. Paiva}
\address{Universidade Federal de Mato Grosso do Sul, Aquidauana, MS, 79200-000}
\email{thales.paiva@ufms.br}

\subjclass[2020]{Primary 54A20; Secondary 54A05,54C05}
\keywords{convergence spaces, nets, pasting lemma, finite stability.}

\begin{abstract}
We prove a pasting lemma for locally finite closed covers of convergence spaces with limit codomains. We also show how the validity of such lemmas characterizes limit spaces.
\end{abstract}

\maketitle

\section*{Introduction}

The study of convergence spaces as generalizations of topological spaces goes back to Choquet's introduction of \emph{pseudotopologies} and \emph{pretopologies}~\cite{Choquet}. A \emph{convergence space} is a set equipped with a \emph{notion of convergence} for filters (or nets) defined on it. In other words, it provides a way to relate these objects to points in the space, which are called their limit points. Continuity means preservation of convergence. Every topological space determines such a structure, but there are convergence structures that cannot be induced by a topology. Nevertheless, topological notions can be considered beyond the topological setting. For an exposition of the theory, see~\cite{BB,dol, DM,preuss}. One such topological notion suitable for generalization to convergence spaces is the pasting lemma.

The classic pasting lemmas are basic continuity principles allowing continuous functions defined on the members of a cover to be glued into a continuous function on the whole space. In convergence spaces, however, the situation is less clear, as the gluing must take care of the convergent behavior directly.

For pasting over closed covers, Preuss~\cite[Proposition A.2.2, p.~260]{preuss} considers a pretopological domain and a limit codomain, with a cover by two closed subsets. Dossena~\cite[Lemma 3.2, Remark 3.6 and Problem 1]{dossena} formulates the open-cover and locally finite closed-cover alternatives in pseudotopological spaces and asks for their extension to limit spaces\footnote{Dolecki and Mynard~\cite[Corollary IX.10.15, p.~275]{DM} give a version for a convergence domain and a locally finite closed cover, with a pseudotopological codomain.}. The different hypotheses in these results make it natural to ask how far closed pasting can be extended, and whether its validity imposes conditions on the convergence of the codomain.

We prove a pasting lemma for locally finite closed covers with a convergence domain and a limit codomain, and characterize limit spaces by the validity of closed pasting. Thus the pasting principle both holds in this class and determines it among convergence codomains. Our proofs use nets, with alternative proofs using filters included for \emph{comparison}. The homotopical significance of the results is discussed briefly in the final section, while the focus of the paper remains on pasting and its convergence hypotheses.

\section{Nets, filters and convergence spaces}\label{secnets}

In 1922, Moore and Smith~\cite{MooreSmith} introduced nets to generalize sequences, allowing for the handling of various types of limits in Analysis. In modern terminology, a \emph{net} in a set $X$ is a function $\varphi\colon\mathbb{D}\to X$, where $\mathbb{D}$ is a directed set: this means that $\mathbb{D}$ is nonempty and has a binary relation $\preceq$ which is both reflexive and transitive (i.e., it is a \emph{preorder}) such that for every $a,b\in\mathbb{D}$ there is $c\in\mathbb{D}$ for which $a,b\preceq c$. By a \emph{tail set} of $\varphi$ we mean a subset like $\varphi[a^{\uparrow}]$ for some $a\in\mathbb{D}$, where $a^{\uparrow}\coloneqq\{b\in\mathbb{D}:b\succeq a\}$, thus $\varphi[a^{\uparrow}]=\{\varphi(b):b\succeq a\}$.

For a topological space $\apair{X,\tau}$, we say that a net $\varphi$ $\tau$\emph{-converges} to a point $x\in X$ if every $\tau$-open set $O\subseteq X$ containing $x$ also contains a tail set of $\varphi$, i.e., there is $a\in\mathbb{D}$ such that $\varphi(b)\in O$ for every $b\succeq a$. We abbreviate this by writing $\varphi\to_\tau x$ or $\varphi(d)\to_\tau x$.
Sequences are specific cases of nets, and the definition of convergent sequences in Analysis is a particular instance of the general definition of net convergence.
Although simple, the following well-known result helps to motivate the discussion that follows.

\begin{proposition}[Folklore\footnote{Arguably the first appearance of this result is due to Birkhoff~\cite{birkhoff}.}]\label{trigger} For topological spaces $\apair{X,\tau}$ and $\apair{Y,\tau'}$, a function $f\colon X\to Y$ is continuous if and only if $f\circ \varphi$ $\tau'$-converges to $f(x)$ for every $x\in X$ and every net $\varphi$ in $X$ such that $\varphi\to_\tau x$.\end{proposition}

\begin{proof}
If $f$ is not continuous, then there is $V\in\tau'$ such that $f^{-1}[V]\notin \tau$, so there is $x\in f^{-1}[V]$ such that $f[U]\not\subseteq V$ for every $U\in\tau$ for which $x\in U$. By choosing $\varphi(U)\in U$ such that $f(\varphi(U))\notin V$ for all those $U$, we obtain a net $\varphi\colon \tau_x \to X$ such that $\varphi\to_\tau x$ and $f\circ \varphi\mathrel{\not\!\to_{\tau'}} f(x)$, where $\tau_x\coloneqq\{U\in\tau:x\in U\}$ is directed by reverse inclusion. The converse is straightforward.
\end{proof}

Proposition~\ref{trigger} shows that convergent nets in a topological space completely determine its behavior in $\Top$, the category of topological spaces. Therefore, a natural step towards generalization is to consider \emph{abstract convergences}, i.e., \emph{functions} of the form $\NETs{X}\to\wp(X)$, where $\NETs{X}$ represents the ``set'' of all nets in $X$ and $\wp(X)$ is the \emph{power set} of $X$.
 There is, however, a technical problem: if $X\ne \emptyset$, then $\NETs{X}$ is \emph{not} a set, but rather a \emph{proper class}\footnote{Cf.~Jech~\cite[pp.5,6]{Jech2002} and Schechter~\cite[p.169]{schechter}. The class of nets in a nonempty set is proper because their directed domains may have arbitrary cardinality.}.

To handle this, notice that the tail sets of the nets are what truly matter in the topological definition. So, for $\varphi,\psi\in\NETs{X}$, let us say that $\varphi$ and $\psi$ are \emph{tail equivalent} if every tail set of $\varphi$ contains a tail set of $\psi$ and vice versa. This relation, despite potential ZFC complications, defines an equivalence relation on $\NETs{X}$, where the equivalence classes correspond precisely to the filters on $X$.

Filters were introduced by Henri Cartan \cite{cartan1} in 1937 and popularized by Bourbaki. A \emph{proper filter} on $X$ is a nonempty family of nonempty subsets of $X$, closed under supersets and finite intersections. We denote the family of proper filters on $X$ by $\Fil{X}$.
For a topological space $\apair{X,\tau}$, the family $\mathcal{N}_x$ of $\tau$-neighborhoods of $x$ is a typical example of filter, which is used to define topological convergence in terms of filters: a proper filter $\mathcal{F}$ on $X$ is said to \emph{$\tau$-converges} to $x$ if $\mathcal{N}_x\subseteq \mathcal{F}$, what we abbreviate with $\mathcal{F}\to_\tau x$.

Both notions of convergence are related. First, notice that a net $\varphi$ in $X$ induces a filter $\varphi^{\uparrow}$ in a natural way: let $\varphi^{\uparrow}$ be the family of those subsets of $X$ containing a tail set of the net $\varphi$. This yields a correspondence
$(\cdot)^{\uparrow}\colon \NETs{X}\to\Fil{X}$
such that $\varphi\to_\tau x$ if and only if $\varphi^{\uparrow}\to_\tau x$. Notice that two nets are tail equivalent if and only if they induce the same filter.
On the other hand, with a ``right inverse''
$\Gamma\colon\Fil{X}\to\NETs{X}$
we would have $\Gamma(\mathcal{F})\to_\tau x$ if and only if $\left(\Gamma(\mathcal{F})\right)^{\uparrow}=\mathcal{F}$ and $\mathcal{F}\to_\tau x$.

There is, indeed, such a function, as shown by Bruns and Schmidt~\cite{BSTrick}. Here is how it works. Given a filter $\mathcal{F}\in\Fil{X}$, define
$\mathbb{D}_{\mathcal{F}}\coloneqq\left\{\apair{x,F}\in X\times\mathcal{F}:x\in F\right\}$,
preordered by the relation $\leq$ such that $\apair{x,F} \leq \apair{y,G}$ if and only if $G \subseteq F$, which makes $\apair{\mathbb{D}_\mathcal{F}, \leq}$ a directed set. In this setup, the net $\Gamma(\mathcal{F})\colon\mathbb{D}_\mathcal{F}\to X$, defined by projection onto the first coordinate, has the members of $\mathcal{F}$ as its tail sets, ensuring that $\left(\Gamma(\mathcal{F})\right)^{\uparrow} = \mathcal{F}$.

Since every net induces a proper filter and every proper filter is induced by a canonical net via $\Gamma$, the correspondences $(\cdot)^{\uparrow}$ and $\Gamma$ allow us to freely switch between nets and filters when discussing convergence in topological contexts. This is well known already. But we can go even further.

Typically, a (filter) \emph{preconvergence} on a set $X$ is a function $L\colon \Fil{X} \to \wp(X)$. However, by using $(\cdot)^{\uparrow}$, we can extend $L$ to a ``function'' $L\colon \NETs{X} \to \wp(X)$ compatible with tail equivalence as follows.
A (net) \emph{preconvergence} on a  set $X$ is a ``function'' $L:\NETs{X}\to\wp(X)$ such that $L(\varphi)=L(\psi)$ whenever $\varphi^{\uparrow}=\psi^{\uparrow}$.

Since (net) preconvergences do not distinguish nets inducing the same filters, it becomes clear that nets here enter simply as a way to parametrize filters in the classical definition\footnote{Equivalently, the filter-based and net-based categories of preconvergence spaces are concretely isomorphic. For a discussion of nets and their relations to filters in the context of convergence theory, see~\cite{schechter,Obrien}.}.

As usual, we write $\varphi\to_L x$ instead of $x\in L(\varphi)$, and say that $x$ is an \emph{$L$-limit} of $\varphi$, or that $\varphi$ \emph{$L$-converges to $x$}, etc. The pair $\apair{X,L}$ is referred to as a \emph{preconvergence space}.
A function \( f \colon X \to Y \) between preconvergence spaces \( \apair{X, L} \) and \( \apair{Y, L'} \) is said to be \emph{continuous} if the image of \( L(\varphi) \) under \( f \) is contained in \( L'(f \circ \varphi) \) for every \( \varphi \in \NETs{X} \), that is,
$
f\left[L(\varphi)\right] \subseteq L'(f \circ \varphi).
$
Equivalently, \( f \circ \varphi \to_{L'} f(x) \) whenever \( \varphi \to_L x \).

Now, we say a preconvergence $L$ on $X$ is
\begin{enumerate}
\item \emph{centered} if constant nets $L$-converge to their corresponding constants,
\item \emph{isotone} if $L(\varphi)\subseteq L(\psi)$ whenever $\varphi^\uparrow\subseteq \psi^{\uparrow}$,
\item \emph{stable} if $L(\varphi)\cap L(\psi)\subseteq L(\rho)$,
\end{enumerate}
where $\varphi$ and $\psi$ are nets in $X$, and $\rho$ is an arbitrary \emph{mixing} of $\varphi$ and $\psi$. The details of these conditions are briefly explained below.

\emph{Centeredness} should be clear: since we want constant functions to be continuous, we require constant nets to converge. \emph{Isotonicity}, on the other hand, abstracts the behavior of subsequences in typical topological contexts. Notice that if $\apair{y_n}_n$ is a subsequence of $\apair{x_n}_n$, then $\apair{x_n}_n^\uparrow\subseteq \apair{y_n}_n^{\uparrow}$. Thus, in an isotone preconvergence, if $\apair{x_n}_n\to x$ and $\apair{y_n}_n$ is a subsequence of $\apair{x_n}_n$, we must have $\apair{y_n}_n\to x$ as well. With this in mind, we follow Schechter~\cite{schechter} and say that $\psi$ is a \emph{subnet} of $\varphi$ if $\varphi^\uparrow\subseteq \psi^{\uparrow}$. Isotonicity, therefore, ensures that convergence is preserved for subnets. A subset $E\subseteq\mathbb D$ is called \emph{cofinal} if for every $d\in\mathbb D$ there exists $e\in E$ such that $d\preceq e$. In this case, $E$ is directed with its induced order, and $\varphi|_E$ is a subnet of $\varphi$. As Dolecki and Mynard~\cite{DM}, we say $\apair{X,L}$ is a \emph{convergence space} if $L$ is both centered and isotone, and their category is denoted by $\Conv$.

\begin{remark}\label{cofinaltrick} As $\varphi^\uparrow = \left(\varphi(d):d\succeq D\right)^\uparrow$ for every $D\in\mathbb{D}$, we can always restrict a net to some of its tail sets without losing any of its limit points. We will use this observation in the proof of Proposition~\ref{openpasting}.\end{remark}

Finally, \emph{stability} refers to the ability to mix two sequences converging to the same point to form a new sequence that also converges to that point\footnote{Stability is also called \emph{finite depth} in~\cite{DM}, which in filter terminology means that $L(\mathcal{F})\cap L(\mathcal{G})\subseteq L(\mathcal{F}\cap\mathcal{G})$ for every proper filters $\mathcal{F}$ and $\mathcal{G}$ on $X$. It is equivalent to the net version presented here provided the preconvergence $L$ is isotone. We refer the reader to O'Brien et al.~\cite{Obrien} for a proof, as well as other notions of mixings that we will not use in this work.}. For instance, if $\apair{x_n}_n$ and $\apair{y_n}_n$ are sequences in a topological space converging to a point $z$, then $\apair{z_n}_n\to z$, where $z_n=x_n$ for even $n$ and $z_n=y_n$ for odd $n$.
Mixing of nets generalizes this situation. Given two nets $\varphi,\psi\in\NETs{X}$ with the same domain $\mathbb{D}$, a net $\rho\colon\mathbb{D}\to X$ is called a \emph{mixing} of $\varphi$ and $\psi$ if there exists $D\in\mathbb{D}$ such that $\rho(d)\in\{\varphi(d),\psi(d)\}$ for every $d\succeq D$. Stability ensures that if $\rho$ is a mixing of the nets $\varphi$ and $\psi$, both converging to $x$, then $\rho$ also converges to $x$. A convergence space $\apair{X,L}$ is called a \emph{limit space} if $L$ is stable, and the category of such spaces is denoted by $\Lim$.
Following Preuss's terminology~\cite{preuss}, a convergence space $\apair{X,L}$ is called a \emph{Kent space} if condition~(3) holds at $x$ whenever $\varphi\to_L x$ and $\psi=\langle x\rangle$ is the constant net at $x$ on the same directed domain as $\varphi$. The category of Kent spaces is denoted by $\Kent$.

% A finite mixing is defined analogously. Binary stability implies finite stability by induction.

It should be clear that $\Top\subseteq\Lim\subseteq\Kent\subseteq\Conv$. For examples illustrating the failure of the reverse inclusions, as well as further discussions about ``intermediate'' categories between $\Top$ and $\Lim$, we refer the reader to~\cite{dol} and~\cite{DM}. Pretopological and pseudotopological spaces are important subclasses of limit spaces~\cite{BB,DM} that will not be explicitly used in our proofs.

As in topology, restrictions of functions use the subspace structure. A subset $X$ of a preconvergence space $\apair{Y,L}$ inherits the \emph{subspace preconvergence} $L|_X$, in which a net $\psi\in\NETs{X}$ $L|_X$-converges to $x\in X$ if and only if $\psi\to_L x$ in $Y$.

Another familiar translation concerns openness and closure. A subset $O\subseteq X$ of a preconvergence space $\langle X,L\rangle$ is said to be $L$-\emph{open} if every net converging to a point of $O$ \emph{eventually} lies in $O$, that is, $O$ contains a tail set of the net. The family $\tau_L$ of all $L$-open sets is a topology on $X$. We denote it by $\tau X$, the \emph{topological modification}~\cite{BB}. $L$-closed sets are the complements of $L$-open sets.

In an isotone space, a subset $C$ is closed when every net in $C$ has all its limits in $C$. For the reverse implication, if a net with limit outside $C$ were not eventually outside $C$, its restriction to the cofinal set of indices at which it lies in $C$ would contradict this condition. Continuous maps pull back open and closed sets to open and closed sets, directly by the net criterion.

With respect to $\tau X$, a family $\mathcal C$ of subsets of $X$ is \emph{locally finite} if for every \(x \in X\), there exists an open set \(O\) such that \(x \in O\) and the set \(\{C \in \mathcal{C} : C \cap O\ne \emptyset\}\) is finite. In particular, this condition is automatically satisfied when \(\mathcal{C}\) is finite.

\section{The pasting lemmas}\label{secpasting}

We present two versions of the pasting lemmas for locally finite closed covers. The first assumes a Kent domain and a limit codomain. The second allows an arbitrary convergence domain and a limit codomain. The first version has a visually simpler proof. It was only after obtaining this proof using nets that we saw how to generalize the result. In our view, this justifies presenting both proofs, even though the second version generalizes the first.

We begin with the familiar open-cover version. Dossena~\cite[Lemma 3.2]{dossena} states it for pseudotopological spaces. We include its proof to introduce the way of arguing with nets.

\begin{proposition}[Open Pasting Lemma]\label{openpasting}
Let $X,Y$ be convergence spaces, let $\mathcal O$ be an open cover of $X$, and let $f\colon X\to Y$ be a function. If $f|_O$ is continuous for each $O\in\mathcal O$, then $f$ is continuous.
\end{proposition}
\begin{proof}
For $\varphi\to x$, choose $O\in\mathcal O$ containing $x$. The net is eventually in $O$, and its tail converges to $x$ in the subspace $O$. Continuity of $f|_O$ gives convergence of the image tail to $f(x)$. This tail is tail equivalent to $f\circ\varphi$ (cf.~Remark~\ref{cofinaltrick}), proving the assertion.
\end{proof}

The closed cover case we shall see below is more involved: it recombines finitely many image nets, and this is where the target's stability enters. Figure~\ref{fig:pasting} illustrates this construction.

\begin{theorem}[Pasting lemma for limit spaces]\label{pastinglemma} Let $X$ be a Kent space and $Y$ be a limit space, and let $f\colon X\to Y$ be a function. If $\mathcal{C}$ is a locally finite cover for $X$ by closed sets such that the restriction $f|_C$ is continuous for every $C\in\mathcal{C}$, then $f$ is continuous.\end{theorem}

\begin{proof}
For a point $x\in X$ and a net $\varphi$ in $X$ such that $\varphi\to x$, we need to show that $f\circ \varphi\to f(x)$ in $Y$. In order to do this, let $O\subseteq X$ be an open set such that $x\in O$ and $\mathcal{S}\coloneqq\{C\in\mathcal{C}:C\cap O\ne\emptyset\}$ is finite. Since each $S\in\mathcal{S}$ is closed and $\mathcal{S}$ is finite, there is no loss of generality in assuming $x\in S$ for every $S\in\mathcal{S}$: otherwise,
\[O'\coloneqq O\cap \left(X\setminus \bigcup\left\{S\in\mathcal{S}:x\notin S\right\}\right)\]
is an open set such that $x\in O'$ and $\mathcal{S}'\coloneqq\{C\in\mathcal{C}: O'\cap C\ne\emptyset\}$ is finite and $x\in S$ for every $S\in\mathcal{S}'$.

Now, let $\varphi_S\colon\dom{\varphi}\to X$ be the net given by the rule
\[\varphi_S(d)\coloneqq \begin{cases}
\varphi(d)&\text{if }\varphi(d)\in S\\
\quad x&\text{if }\varphi(d)\not\in S\end{cases}.\]

Since $X$ is a Kent space, it follows that $\varphi_S\to x$. Given that $\varphi_S$ is a net in $S$ and $x\in S$, we get $\varphi_S\to x$ in the subspace $S$, so $f\circ \varphi_S\to f(x)$ in $Y$ due to the continuity of $f|_S$. As $Y$ is a limit space, to conclude that $f\circ \varphi\to f(x)$, it suffices to find $A\in\dom{\varphi}$ such that $f(\varphi(\alpha))\in\{f(\varphi_S(\alpha)):S\in\mathcal{S}\}$ for every $\alpha\succeq A$. Since $\varphi\to x$ and $x\in O$, there exists $A\in\dom{\varphi}$ such that $\varphi(\alpha)\in O$ for every $\alpha\succeq A$. Finally, notice that as $\mathcal{C}$ covers $X$, for every $\alpha\succeq A$ there is $C\in\mathcal{C}$ such that $\varphi(\alpha)\in C$, hence $C\cap O\ne\emptyset$, thus implying $C\in\mathcal{S}$ and $\varphi_C(\alpha)=\varphi(\alpha)$, so $f(\varphi(\alpha))=f(\varphi_C(\alpha))$, as desired.
\end{proof}

\begin{figure}[!htbp]
\centering
\includegraphics[width=.63\textwidth]{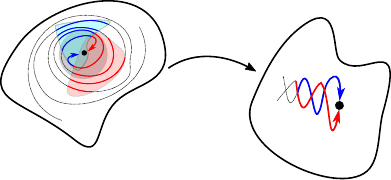}
\caption{The pasting method of the previous proof for two closed sets.}
\label{fig:pasting}
\end{figure}
\FloatBarrier

\begin{remark}[Alternative proof via filters]\label{filterproof}
For readers who are ideologically opposed to nets (see Dolecki's discussion in~\cite[Section~VII.6, pp.~143--145]{dol}), we present an alternative proof of Theorem~\ref{pastinglemma} using filters. This comparison assumes familiarity with convergence and continuity in terms of filters. For the definitions and the translations between nets and filters, see~\cite{dol,Obrien}.

In filter notation, stability says that finite intersections of filters converging to the same point still converge to that point. Isotonicity says that finer filters preserve limits. For a map $h$, write $h(\mathcal F)=\{B:h^{-1}(B)\in\mathcal F\}$ for the direct image filter.

Let $\mathcal F\to x$ and put $\mathcal H=\mathcal F\cap\{x\}^{\uparrow}$, where $\{x\}^{\uparrow}$ is the principal filter at $x$. Since $X$ is a Kent space, we have $\mathcal H\to x$. As in the preceding proof, choose an open neighborhood $O$ of $x$ meeting only $C_1,\ldots,C_m\in\mathcal C$, with $x\in C_i$ for every $i$. Then $O\in\mathcal H$ and $O\subseteq C_1\cup\cdots\cup C_m$.

For each $i$, let $\mathcal H_i$ be the filter on $X$ generated by $\{A\cap C_i:A\in\mathcal H\}$. It is proper because every generator contains $x$. Since $\mathcal H_i\supseteq\mathcal H$, we have $\mathcal H_i\to x$. Its trace on $C_i$ converges to $x$ in the subspace, so continuity of $f|_{C_i}$ gives $f(\mathcal H_i)\to f(x)$.

We have $\mathcal H=\bigcap_{i=1}^m\mathcal H_i$. Indeed, if $B\in\bigcap_i\mathcal H_i$, choose $A_i\in\mathcal H$ with $A_i\cap C_i\subseteq B$. Then $O\cap\bigcap_i A_i\subseteq B$, so $B\in\mathcal H$. The reverse inclusion follows from $\mathcal H\subseteq\mathcal H_i$. Direct images preserve intersections, and hence
\[
 f(\mathcal H)=\bigcap_{i=1}^m f(\mathcal H_i)\to f(x)
\]
by stability in $Y$. Finally, $\mathcal F\supseteq\mathcal H$, so isotonicity gives $f(\mathcal F)\to f(x)$.
\end{remark}

\begin{theorem}\label{generaldomain}
Let $X$ be a convergence space, let $Y$ be a limit space, and let $\mathcal C$ be a locally finite cover of $X$ by closed subsets. If a function $f\colon X\to Y$ has continuous restriction $f|_C$ for every $C\in\mathcal C$, then $f$ is continuous.
\end{theorem}
\begin{proof}
Let $\varphi\colon\mathbb D\to X$ converge to $x$. Local finiteness gives a tail lying in a finite union $C_1\cup\cdots\cup C_m$ of closed cover members. Put $E_i=\{d\in\mathbb D:\varphi(d)\in C_i\}$. The noncofinal $E_i$ can all be avoided on a common tail, and at least one $E_i$ is cofinal, since a finite union of noncofinal sets cannot contain a tail.

For each cofinal $E_i$, choose $r_i\colon\mathbb D\to E_i$ with $r_i(d)\succeq d$ and $r_i(d)=d$ on $E_i$. All these maps use the original domain $\mathbb D$. Then $\varphi_i=\varphi\circ r_i$ is a subnet of $\varphi$, so $\varphi_i\to x$ by isotonicity. Since its values lie in the closed set $C_i$, we have $x\in C_i$, and continuity of $f|_{C_i}$ gives $f\circ\varphi_i\to f(x)$. On the common tail, each value of $f\circ\varphi$ agrees at the same index with one of these finitely many image nets. Stability in $Y$ gives the conclusion. Our subnet convention does not require the $r_i$ to be monotone. This argument shows that stability is needed only in the codomain.
\end{proof}

\begin{remark}[Alternative proof via filters]\label{filterproofgeneraldomain}
Let $\mathcal F\to x$. Choose an open neighborhood $O$ of $x$ meeting only $C_1,\ldots,C_m\in\mathcal C$, so that $O\in\mathcal F$ and $O\subseteq\bigcup_{i=1}^m C_i$. Put
\[
 J=\{i:A\cap C_i\ne\emptyset\text{ for every }A\in\mathcal F\}.
\]
For each $i\notin J$, choose $A_i\in\mathcal F$ disjoint from $C_i$. Then $A_0=O\cap\bigcap_{i\notin J}A_i\in\mathcal F$ and $A_0\subseteq\bigcup_{j\in J}C_j$. Since $\mathcal F$ is proper, $J$ is nonempty.

For $j\in J$, the sets $A\cap C_j$, with $A\in\mathcal F$, generate a proper filter $\mathcal F_j$ on $X$. It is finer than $\mathcal F$, so $\mathcal F_j\to x$ by isotonicity. Since $C_j\in\mathcal F_j$ and $C_j$ is closed, $x\in C_j$. The trace of $\mathcal F_j$ on $C_j$ has image $\mathcal F_j$ under inclusion, hence converges to $x$ in that subspace. Continuity of $f|_{C_j}$ gives $f(\mathcal F_j)\to f(x)$.

We have $\mathcal F=\bigcap_{j\in J}\mathcal F_j$. Indeed, one inclusion follows from $\mathcal F\subseteq\mathcal F_j$. For the other, if $B\in\bigcap_{j\in J}\mathcal F_j$, choose $B_j\in\mathcal F$ with $B_j\cap C_j\subseteq B$. Then $A_0\cap\bigcap_{j\in J}B_j\subseteq B$, so $B\in\mathcal F$. Using the direct image notation from Remark~\ref{filterproof}, we conclude that
\[
 f(\mathcal F)=\bigcap_{j\in J}f(\mathcal F_j)\to f(x)
\]
by finite stability in $Y$.
\end{remark}

For a finite closed cover, local finiteness is automatic. Theorem~\ref{generaldomain} extends Preuss's two-piece statement~\cite[Proposition A.2.2, p.~260]{preuss}, which assumes a pretopological source. In particular, Proposition~\ref{openpasting} and Theorem~\ref{pastinglemma}, restricted to $\Lim$, give the extension asked for in Dossena's Problem~1~\cite{dossena}.

The target assumption cannot simply be weakened to Kent convergence. The following example is Preuss's counterexample expressed in net language.

\begin{example}[{\cite[Remark A.2.3, pp.~260--261]{preuss}}]\label{preussexample}
Put $I=[0,1]$, $p=1/2$, $A=[0,p]$ and $B=[p,1]$. Give $I$ a convergence $L$ which agrees with convergence in the usual topology at every point other than $p$, and declare $\varphi\to_L p$ precisely when $\varphi$ converges in the usual topology to $p$ and is eventually in $A$ or eventually in $B$. This is centered and isotone. The resulting convergence space is Kent: replacing terms by the limit preserves convergence in the usual topology, and at $p$ it also preserves the chosen half-interval. However, the sequences $p-1/(n+3)$ and $p+1/(n+3)$ both converge to $p$, whereas their alternating mixing does not $L$-converge to $p$.

The identity maps from the subspaces $A,B$ with their usual topologies to $(I,L)$ are continuous. Their pasting, the identity on the interval with its usual topology, is not continuous, as the same alternating sequence shows. Equivalently, the paths $t\mapsto t/2$ and $t\mapsto(t+1)/2$ are composable, while their concatenation is this discontinuous identity.
\end{example}

\section{Stability as the validity of gluing functions}\label{seccharacterization}

We now characterize the codomains for which closed pasting holds. Within Kent spaces, topological domains covered by two closed subsets suffice to characterize limit spaces. For arbitrary convergence codomains, we will allow convergence spaces as domains. We first need a lemma.

Let $\mathbb D$ be a directed set. On $\mathbb D^\infty=\mathbb D\cup\{\infty\}$, isolate the points of $\mathbb D$ and take $\{\infty\}\cup d^\uparrow$, for $d\in\mathbb D$, as a neighborhood basis at $\infty$. A net in $\mathbb D^\infty$ converges to $\infty$ if and only if it is eventually in $\{\infty\}\cup d^\uparrow$ for every $d\in\mathbb D$.

\begin{lemma}\label{netextension}
For a Kent space $Y$, a net $\varphi\colon\mathbb D\to Y$ converges to $y$ if and only if its extension $\varphi^\infty$, with $\varphi^\infty(\infty)=y$, is continuous.
\end{lemma}
\begin{proof}
Directedness makes the indicated sets a neighborhood basis. Suppose $\varphi\to y$. On $\mathbb D\times\{0,1\}$ use the preorder $(d,i)\preceq(e,j)$ exactly when $d\preceq e$. The net $\widetilde\varphi(d,i)=\varphi(d)$ is tail equivalent to $\varphi$. Let $\langle y\rangle$ be the constant net with value $y$ on the same directed set, and define their mixing by
\[
 \widehat\varphi(d,i)=
 \begin{cases}
 \widetilde\varphi(d,0),&i=0,\\
 \langle y\rangle(d,1),&i=1.
 \end{cases}
\]
Since $Y$ is a Kent space, we have $\widehat\varphi\to y$. Its tail sets are exactly $\{y\}\cup\varphi[d^\uparrow]$.

If $\eta\to\infty$ in $\mathbb D^\infty$, its image $\varphi^\infty\circ\eta$ is eventually in $\{y\}\cup\varphi[d^\uparrow]$ for every $d$. It is therefore a subnet of $\widehat\varphi$ and converges to $y$ by isotonicity. At an isolated point of $\mathbb D$, continuity follows from eventual constancy and centeredness. Conversely, the net $d\mapsto d$ converges to $\infty$ in $\mathbb D^\infty$. Its image under a continuous extension is $\varphi$, proving $\varphi\to y$.
\end{proof}

\begin{samepage}
\begin{theorem}\label{characterization}
For a Kent space $Y$, the following conditions are equivalent:
\begin{enumerate}
\item $Y$ is a limit space.
\item for every topological space $T$ covered by two closed subsets $T_0,T_1$, a function $F\colon T\to Y$ is continuous whenever $F|_{T_0}$ and $F|_{T_1}$ are continuous.
\end{enumerate}
\end{theorem}
\end{samepage}
\begin{proof}
Theorem~\ref{pastinglemma} gives $(1)\Rightarrow(2)$, since topological spaces are limit spaces. To prove $(2)\Rightarrow(1)$, let $\varphi,\psi\colon\mathbb D\to Y$ converge to $y$. Form
\[
 T=(\mathbb D\times\{0,1\})\sqcup\{\infty\},
\]
isolating every point other than $\infty$ and taking
\[
 U_d=\{\infty\}\cup(d^\uparrow\times\{0,1\})
\]
as basic neighborhoods of $\infty$.

The branches $T_i=(\mathbb D\times\{i\})\cup\{\infty\}$ are closed and have the topology of $\mathbb D^\infty$ in Lemma~\ref{netextension}. Define $F(d,0)=\varphi(d)$, $F(d,1)=\psi(d)$ and $F(\infty)=y$. That lemma makes the restrictions continuous, so (2) makes $F$ continuous. Given a mixing $\rho$ of $\varphi$ and $\psi$, choose labels $i_d$ on a tail so that $\rho(d)$ is the selected value, and choose labels arbitrarily before it. The net $d\mapsto(d,i_d)$ converges to $\infty$. Its image under $F$ is tail equivalent to $\rho$, proving $\rho\to y$ and therefore stability.
\end{proof}

\begin{example}\label{nonKentcounterexample}
The Kent hypothesis is needed for the characterization by topological domains in Theorem~\ref{characterization}. Let $Y=\{0,1\}$, declaring $\varphi\to0$ precisely when $\varphi$ is eventually constant at either point, and $\varphi\to1$ precisely when $\varphi$ is eventually equal to $1$. This convergence is centered and isotone. It is not Kent: the constant sequence $1$ converges to $0$, but replacing alternate terms by $0$ gives a sequence which does not converge to $0$. In particular, $Y$ is not a limit space.

A map $f\colon T\to Y$ from a topological space is continuous if and only if it is locally constant. Centeredness gives one implication. Conversely, if $f$ is not locally constant at $x$, choose $x_U\in U$ with $f(x_U)\ne f(x)$ for every open neighborhood $U$ of $x$. Direct the pairs $(U,j)$, with $j\in\{0,1\}$, by reverse inclusion of $U$, ignoring the second coordinate. The net with values $x$ at $(U,0)$ and $x_U$ at $(U,1)$ converges to $x$, while every tail of its image contains both points of $Y$. This contradicts continuity.

Now let $T=A\cup B$ be a topological space covered by two closed subsets. If $f|_A$ and $f|_B$ are continuous into $Y$, they are continuous into the discrete two-point space. Theorem~\ref{pastinglemma} gives continuity of $f$ into that discrete space, hence into $Y$. Thus $Y$ satisfies condition~(2) of Theorem~\ref{characterization}, but fails condition~(1).
\end{example}

Limit spaces can also be characterized among all convergence spaces by allowing convergence spaces as domains.

\begin{proposition}\label{convcharacterization}
For a convergence space $Y$, the following conditions are equivalent:
\begin{enumerate}
\item $Y$ is a limit space.
\item closed pasting into $Y$ holds for every cover by two closed subsets of every convergence space.
\end{enumerate}
\end{proposition}
\begin{proof}
Theorem~\ref{generaldomain} gives $(1)\Rightarrow(2)$. For $(2)\Rightarrow(1)$, let $\varphi,\psi\colon\mathbb D\to Y$ converge to $y$. Put $X=(\mathbb D\times\{0,1\})\sqcup\{\infty\}$ and $X_i=(\mathbb D\times\{i\})\cup\{\infty\}$. Give $\mathbb D\times\{0,1\}$ the preorder comparing only the first coordinate, and let $u$ be its inclusion into $X$. Declare convergence to $a\ne\infty$ to mean eventual equality to $a$, and declare $\eta\to\infty$ if and only if $\eta$ is eventually equal to $\infty$ or $u^\uparrow\subseteq\eta^\uparrow$. These rules are centered, isotone and tail invariant, so $X$ is a convergence space. Both $X_i$ are closed: a limit outside a piece would require eventual equality to a point outside it.

Define $F(d,0)=\varphi(d)$, $F(d,1)=\psi(d)$ and $F(\infty)=y$. Each restriction is continuous. At points other than $\infty$ this follows from centeredness, as it does for nets eventually equal to $\infty$. In the remaining case, a net in $X_i$ converging to $\infty$ is eventually in $d^\uparrow\times\{i\}$ for every $d$. Its image is a subnet of $\varphi$ or $\psi$, so isotonicity gives its limit $y$. By (2), $F$ is continuous. For an eventual mixing $\rho$, choose branch labels $i_d$ realizing it on a tail, and choose them arbitrarily elsewhere. The net $d\mapsto(d,i_d)$ is a subnet of $u$, hence converges to $\infty$. Its $F$-image agrees eventually with $\rho$, proving stability.
\end{proof}

\section{Further comments}\label{secconsequences}

Taken together, Theorem~\ref{generaldomain} and Proposition~\ref{convcharacterization} show that, among convergence spaces, limit spaces are exactly the codomains for which closed pasting over locally finite covers holds for every convergence-space domain. This is the sense in which they form the largest class of targets for this pasting principle.

These results may also be useful in algebraic topology. Concatenating paths and homotopies involves compatible definitions on closed subintervals of the parameter interval. Pasting supplies the continuity required for these constructions~\cite{BrownBook}. Their validity depends on the convergence structure, so extending the pasting principle is a basic step toward developing these techniques in more general spaces. Limit spaces are a natural setting for these constructions because of finite stability. Dossena's Problem~1~\cite{dossena} asks whether his closed-cover pasting lemma for pseudotopological spaces extends to limit spaces. Theorem~\ref{generaldomain} gives that extension. His Remark~3.6 also raises the separate question of a fundamental group enriched over $\Lim$, which is not settled by the pasting statement alone.

The development of algebraic topology in convergence spaces has several antecedents. Preuss~\cite[Appendix A.2, especially A.2.2--A.2.7, pp.~260--262]{preuss} develops fundamental groups and higher homotopy groups for limit spaces. Rieser~\cite[Section~3.2.1]{riesernew} establishes homotopy gluing for the standard interval cylinder in limit and pseudotopological spaces. That construction concerns the gluing of cylinders, whereas the theorems here apply to arbitrary locally finite closed covers. The first arXiv version of our article~\cite{MezabarbaMonteiroPaiva2024v1} develops the fundamental groupoid and a groupoid Seifert--van Kampen theorem using nets. Developing these homotopical topics is no longer the focus of the present paper.

\section*{Acknowledgments}

We would like to thank Carolina Pereiro, Germ\'an Gomero, Hugo C. Bot\'os, Peter Wong, and Weslem Silva for insightful and helpful discussions on the subject. We are also grateful to Antonio Rieser for his thoughtful observations.

\bibliographystyle{amsplain}
\bibliography{refs}

\end{document}